\documentclass[11pt]{amsart}
\usepackage[english]{babel}
\usepackage{amssymb}
\usepackage{enumitem}
\usepackage[initials]{amsrefs}
\usepackage[all]{xy}
\SelectTips{cm}{}

\usepackage{graphicx}
\usepackage{tikz-cd}

\newcommand{\bbN}{{\mathbb N}}

\newcommand{\bbR}{{\mathbb R}}
\newcommand{\bbZ}{{\mathbb Z}}
\newcommand{\bbC}{{\mathbb C}}

\newcommand{\bbT}{\mathbb{T}}

\newcommand{\supp}{\operatorname{supp}}

\newcommand{\Isom}{\operatorname{Isom}}

\newcommand{\Ker}{\operatorname{Ker}}

\newcommand{\Targ}{\operatorname{Targ}}
\newcommand{\Boxx}{\operatorname{Box}}

\newcommand{\half}{\frac{1}{2}}

\newcommand{\SL}{\operatorname{SL}}

\newcommand{\setdef}[2]{{\left\{ {#1} \ :\ {#2}\right\}}}

\newtheorem{mthm}{Theorem}

\newtheorem{theorem}{Theorem}[section]
\newtheorem{lemma}[theorem]{Lemma}

\newtheorem{cor}[theorem]{Corollary}

\newtheorem{prop}[theorem]{Proposition}

\theoremstyle{definition}
\newtheorem{definition}[theorem]{Definition}
\newtheorem{defn}[theorem]{Definition}

\newtheorem{example}[theorem]{Example}

\newtheorem{remark}[theorem]{Remark}

\begin{document}

\title{Diophantine properties of groups of toral automorphisms}

 \author{Vladimir Finkelshtein}

\begin{abstract}
	We prove sharp estimates in a shrinking target problem for the
action of an arbitrary subgroup $\Gamma$ of $\SL_2(\mathbb{Z})$ on the $2-$torus. This can also be viewed as a non-commutative Diophantine approximation
problem. The methods require constructing spectrally optimal
random walks on groups acting properly cocompactly on Gromov
hyperbolic spaces. Additionally, using Fourier analysis we give estimates for the same problem in higher dimensions.
\end{abstract}

\maketitle

\section{Introduction and Statement of Main Results}

This paper studies a certain shrinking target problem for a group $\Gamma<\SL_d(\bbZ)$ with its natural action by automorphisms on the torus $\bbT^d=\bbR^d/\bbZ^d$. Specifically, given a subgroup $\Gamma<\SL_d(\bbZ)$ we are interested in finding infinitely many solutions 
\begin{equation*}
	\setdef{ g\in \Gamma}{ \|g.x- y\|<\psi(\|g\|)}
\end{equation*}	
for a monotonically decreasing function $\psi:\bbR_+\to \bbR_+$, e.g. $\psi(R)=R^{-\alpha}$. Here, $\|x-y\|$ is the distance coming from the Euclidean norm on $\bbR^d$, and $\|g\|$ is the corresponding operator norm on $\SL_d(\bbR)$.
This can be viewed as analogous to the classical (inhomogeneous) Diophantine approximation problem that is concerned with finding solutions
to $\|q.x-y\|<\psi(|q|)$ for $q\in \bbN$ that acts by endomorphisms of $\bbT^1=\bbR/\bbZ$.

\subsection{Approximation by Lebesgue a.e. points on the two torus}\hfill{}\par
We start with the discussion of the $\Gamma$-Diophantine properties of Lebesgue almost every point. Our results have a sharp form in dimension $d=2$. Recall that $\SL_2(\bbZ)$ acts on the hyperbolic plane $(\mathbf{H}^2, d_{\mathbf{H}^2})$. Fix a point $x_0\in \mathbf{H}^2$, and denote
\[
	B_n =\setdef{g\in \Gamma}{ d(g.x_0,x_0)\le n},\qquad
	\delta_\Gamma=\limsup_{n\to\infty} \frac{1}{n}\cdot\log \# B_n.
\]
$\delta_\Gamma$ is the critical exponent of $\Gamma$. 
\begin{mthm}\label{maintheorem} 
Let $\Gamma<\SL_2(\bbZ)$ be an arbitrary subgroup. For any $y\in \bbT^2$, for Lebesgue a.e. $x\in \bbT^2$, 		the set 
\[ \setdef{g\in\Gamma}{\| g.x - y \|<\|g\|^{-\alpha}} \qquad\textrm{is} \] 
	\begin{enumerate}
		\item   finite for every $\alpha> \delta_\Gamma$,
		\item   infinite for every $\alpha< \delta_\Gamma$.
	\end{enumerate}
\end{mthm}
\begin{remark}
Similar shrinking target problems were previously studied by multiple authors. Laurent and Nogueira in \cite{LaurentNogueira} found sharp approximation rates for $\SL_2(\bbZ)$ action on the plane by explicitly constructing the solutions. Less sharp bounds were given for lattices of $\SL_2(\bbR)$ acting on the plane by Maccourant and Weiss in \cite{Maccourant-Weiss} and 
for the $\SL_2(\bbC)$ action on  the complex plane by Policott in \cite{Policott} using effective equidistribution results. Ghosh, Gorodnik and Nevo in \cite{Ghosh-Gorodnik-Nevo} considered a more general setting, where a lattice $\Gamma<G$ acts on a homogenous space $G/H$, with a dense $\Gamma$-orbit. As a corollary, they established Theorem~\ref{maintheorem} for $\Gamma=\SL_2(\bbZ)$ for a.e $y\in \bbT^2$. All of the listed results assumed the acting group to be a lattice. So the main novelty in our work is in treating arbitrary, in particular, \textbf{thin subgroups} $\Gamma$ in $\SL_2(\bbZ)$.
\end{remark}

The proof proceeds via a reduction to subgroups of $\SL_2(\bbZ)$ whose action on the hyperbolic plane is convex cocompact. For such groups we have even sharper estimates below. 
Let us replace the balls around $y\in\bbT^2$ by an arbitrary monotonic family of targets $\{ \Targ_r\}_{r>0}$ of Lebesgue subsets of the torus with measure $m(\Targ_r)=\pi r^2$.
 
\begin{mthm}\label{maintheoremB} 
Let $\Gamma<\SL_2(\bbZ)$ be a subgroup whose action on the hyperbolic plane is convex cocompact. Let $\{ \Targ_r\}_{r>0}$ be a monotonic family of Lebesgue subsets of the torus of measure $m(\Targ_r)=\pi r^2$. Let $\psi:\bbR_+\to \bbR_+$ be a decreasing function.
Then for Lebesgue-a.e. $x\in \bbT^2$ the set 
\[
	\setdef{ g\in \Gamma}{ g.x \in \Targ_{\psi(\|g\|)}} \qquad\textrm{is}
\]
\begin{enumerate}
\item finite, if 
\[
	\sum_{n=1}^\infty n^{2 \delta_\Gamma -1}\psi(n)^2<\infty, 
\]
\item infinite if 
\[
	\sum_{n=1}^\infty (\log n)^4 n^{-2\delta_\Gamma -1} \psi(n)^{-2}<\infty.
\]
\end{enumerate}
\end{mthm}
\begin{remark}
The rates in Theorem~\ref{maintheoremB} are sharper than in Theorem~\ref{maintheorem}. For example, (1) holds for $\psi(R)=R^{-\delta_\Gamma}\log^{-0.5-\epsilon}R$, while (2) holds for $\psi(R)=R^{-\delta_\Gamma}\log^{2.5+\epsilon}R$  for any $\epsilon>0$.
\end{remark}

\begin{remark}\label{strongerapprox}
There is a strictly stronger version of Diophantine approximation, which also holds in our situation. Namely, if $\psi$ is as in Theorem~\ref{maintheoremB}(2), then for any $y\in \bbT^2$, for Lebesgue a.e. $x\in \bbT^2$, for any sufficiently large $T$, there exists $g \in \Gamma$ satisfying
\[
\| g\| \leq T, \qquad \|g.x-y\| \leq \psi(T)
\]
It can be seen immediately from the proof of Theorem~\ref{maintheoremB}(see Remark~\ref{strongerapproxexplain}).
\end{remark}

The finiteness part follows from the first Borel-Cantelli lemma.
The classical independence assumption in the second Borel-Cantelli lemma is often replaced by decay of correlations conditions. In our case, this role is played by spectral estimates for the $\Gamma$-action on the torus $\bbT^2$ as discussed in \S~\ref{spectralestimates}. 

\subsection{Approximation by Diophantine points}\label{diophpoints}\hfill{}\par
For a subgroup $\Gamma<\SL_d(\bbZ)$, $\alpha>0$, and points $x,y\in\bbT^d$, we say that $y$ admits $(\Gamma,\alpha)$-\textit{fast approximation by} $x$ if
\[
	\setdef{g\in \Gamma}{ \|g.x-y\|<\|g\|^{-\alpha}}\qquad\textrm{is\ infinite}.
\]
Theorem~\ref{maintheorem} shows that every $y\in\bbT^2$ is $(\Gamma,\delta_\Gamma-\epsilon)$-fast approximable by Lebesgue a.e. $x\in \bbT^2$ for any $\epsilon>0$. 
In this section we try to analyze how big is the exceptional set of $x\in \bbT^2$, which fail to provide fast approximations for all points $y$ on the torus. We work with subgroups $\Gamma<\SL_d(\bbZ)$ acting on the $d$-torus $\bbT^d$ with $d\ge 2$.

For $q\in \bbN$, denote by $R_q\subset \bbT^d$ the set $\frac{1}{q}\cdot \bbZ^d + \bbZ^d$ -- points with rational coordinates with denominators dividing $q$,
and by $R=\bigcup R_q$ the set of all rational points. We say that a point $x\in\bbT^d$ is $M$-\textit{Diophantine} if there are only finitely many $q\in\bbN$ so that
$x$ is $q^{-M}$-close to a point in $R_q$. Points that are not $M$-Diophantine for any $M$, are called \textit{Liouville}.

Note that rational points $x\in\bbT^d$ have finite $\Gamma$-orbits on the torus (because each $R_q$ is $\SL_d(\bbZ)$-invariant), and so any $y\in \bbT^d \setminus R$ does not admit $(\Gamma,\epsilon)$-fast approximation by $x\in R$ for any $\epsilon>0$.

We want to establish a relation between $(\Gamma,\alpha)-$fast approximability by $x\in \bbT^d$ and Diophantine properties of $x$. We use the results of \cite{BFLM}. 
For $d\ge 3$ we need to impose the following conditions on $\Gamma<\SL_d(\bbZ)$:
\begin{itemize}
	\item[(SI)] $\Gamma$ acts strongly irreducibly on $\bbR^d$, i.e. every subgroup of finite index in $\Gamma$ preserves no non-trivial vector subspaces. 
	\item[(PE)] $\Gamma$ has a proximal element, i.e. an element with a simple dominant eigenvalue.
\end{itemize}
These conditions are automatically satisfied by any $\Gamma<\SL_2(\bbZ)$ with $\delta_\Gamma>0$.

\begin{mthm}\label{wellapproximable}
	Let $\Gamma<\SL_d(\mathbb{Z})$ satisfy (SI) and (PE). Then there exists $C_\Gamma > 0$, 
	such that for every $M>0$, any point $y\in\bbT^d$  is $(\Gamma,\frac{C_\Gamma}{M})$-fast approximable by any $M$-Diophantine point $x\in\bbT^d$.
\end{mthm}

\begin{cor}
	Let $\Gamma$ be as in Theorem~\ref{wellapproximable}. The set of points $x\in \bbT^d$ that do not give $(\Gamma,\epsilon)$-fast approximation 
	of all points $y\in \bbT^d$ for any $\epsilon>0$ consists only of Liouville points and, in particular, has zero Hausdorff dimension.
\end{cor}

\subsection{Spectral estimates}\label{spectralestimates}\hfill{}\par
Let us now state the main spectral estimate needed for the proof of Theorem~\ref{maintheoremB}.
Let $\Gamma<\SL_2(\bbZ)$ be a subgroup whose action on the hyperbolic plane is convex cocompact. 

Let $\pi:\Gamma\to U(\mathcal{H})$ be a unitary $\Gamma$-representation, and $\mu$ a probability measure on $\Gamma$. 
Define the Markov operator on $\mathcal{H}$
\[
	\pi(\mu)=\sum_{g\in\Gamma} \mu(g)\cdot\pi(g).
\]
Note that it always satisfies $\|\pi(\mu)\|\le 1$ and if $\mu$ is symmetric, $\pi(\mu)$ is self-adjoint.
We shall denote by $\pi$ the unitary $\Gamma$-representation on $L^2(\bbT^2)$, and $\pi_0$ the sub-representation on $L^2_0(\bbT^2)$.
In the proof of Theorem~\ref{maintheoremB} we need the estimate provided in the following result.

\begin{theorem}\label{T:specialspec}
There exists a sequence of symmetric probability measures $\mu_n$ on $\Gamma$ with $\supp(\mu_n)\subset B_n$ and
	\[
		\|\pi_0(\mu_n)\|\le  e^{-\half\delta_\Gamma\cdot n + 2\log n +O(1) }.
	\] 
	In fact, the above estimate holds for $\mu_n$ being uniform measures on the shells $S_n=B_n\setminus B_{n-k}$ for some fixed $k$.
\end{theorem}

We can view the spectral estimates we obtained as a quantitative ergodic theorem for the $2-$torus.

\begin{cor}
	Let $\Gamma < \SL_2(\mathbb{Z})$ convex cocompact subgroup, and shells $S_n=B_n\setminus B_{n-k}\subset \Gamma$ as above. 
	Then for any $f\in L^2(\mathbb{T}^2,m)$ we have
	\[
		\left\| \frac{1}{|S_n|} \sum_{g\in S_n} f(g.x) - \int_{\mathbb{T}^2} f dm \right\|_2 
		\leq n^2 e^{-\half\delta_\Gamma\cdot n +O(1) }\cdot \|f\|_2.
	\]
\end{cor}
The constant $\delta_\Gamma$ in the above rate of convergence, cannot be improved(see \S~\ref{spectraloptimality}).

\subsection{Spectrally optimal random walks}\hfill{}\par
For weakly equivalent unitary $\Gamma$-representations $\pi'\sim\pi''$ one has $\|\pi'(\mu)\|=\|\pi''(\mu)\|$ for any probability measure $\mu$ on $\Gamma$. Hence $\pi_0$ in the above theorem can be replaced by any weakly equivalent unitary representation, and we show (Theorem~\ref{toruskoopman}) that the left regular representation $	\lambda:\Gamma\to U(\ell^2\Gamma)
$ is such. So Theorem~\ref{T:specialspec} is a special case of the following more general result, in which convex cocompact subgroup of $\Isom(\mathbf{H}^2)$ is replaced by a group $\Gamma$ acting properly and cocompactly on a proper quasiruled hyperbolic space $(X,d)$. The notion of quasiruled hyperbolic spaces is defined in \S~\ref{background}. We remark that geodesic Gromov hyperbolic spaces are such.  We have the following general form of Theorem~\ref{T:specialspec}: 
\begin{mthm}\label{mainradius}
	Let $(X,d)$ be a proper quasiruled  hyperbolic space, $\Gamma$ a finitely generated group, acting properly cocompacty by isometries on $(X,d)$.
	Then for some $k$ and all $n$, the uniform distributions $\mu_n$ on the shells $S_n=B_n\setminus B_{n-k}$ satisfy
	\[
		\|\lambda(\mu_n)\|\le  e^{-\half\delta_\Gamma\cdot n + 2\log n + O(1)} 
	\]
	where $\lambda$ is the regular representation on $\ell^2(\Gamma)$.
\end{mthm}
In fact, in our proof we replace the regular representation $\lambda$ by the quasi-regular representation on the boundary of $\Gamma$ endowed with the Patterson-Sullivan measure, which satisfies the same estimate. 
\begin{remark}
Similar estimates previously appeared in works of Bader and Muchnik in \cite{BaderMuchnik} and Boyer in \cite{Boyer} in their study of the irreducibility of boundary representations. 
\end{remark}

Let us put Theorem~\ref{mainradius} in a broader perspective. 
Let $\Gamma$ be a group with proper left invariant metric $d$, and let us denote by $B_n$ the ball of radius $n$ in $\Gamma$.
Given a unitary $\Gamma$-representation $\pi$ 
define the function $\rho_\pi : \bbN \to \mathbb{R}_+$ by
\[
	\rho_\pi(n) := \min \setdef{\|\pi(\mu)\|}{\supp(\mu) \subset B_n }.
\]
where $\| \cdot \|$ is the operator norm. Since $B_n\cdot B_m\subset B_{n+m}$, and the operator norm is submultiplicative, one has $\rho_\pi(n+m)\le \rho_\pi(n)\cdot \rho_\pi(m)$. Therefore the limit
\[
	\lim_{n\to\infty} \frac{1}{n}\cdot\log \rho_\pi(n)
\]
exists. It might be of interest to investigate $\rho_\pi(n)$ for a given $\Gamma, d, \pi$ as above.

\par 

For a finitely supported probability measure $\mu$ on $\Gamma$ we recall the definitions of the \textbf{drift} and the \textbf{asymptotic entropy}
\[
	\begin{split}
		\ell(\mu)&:=\lim_{n\to\infty} \frac{1}{n}\cdot \sum_{g\in \Gamma} d(g,e)\cdot\mu^{*n}(g),\\
		h(\mu)&:=\lim_{n\to\infty} \frac{1}{n}\cdot \sum_{g\in \Gamma} -\log \mu^{*n}(g)\cdot \mu^{*n}(g).
	\end{split}
\] 
Let $\lambda$ be the left regular representation of $\Gamma$. The following inequalities are well known and hold for any finitely supported symmetric probability measure
\[  
  -2 \log( 	\| \lambda(\mu)\|) \le  h(\mu) \le \delta_\Gamma\cdot \ell(\mu)
\]  
If $\supp(\mu)\subset B_n$, one has the trivial estimate $\ell(\mu)\le n$, that gives the upper bound $\le \delta_\Gamma \cdot n$
for all of the above. 

Theorem~\ref{mainradius} describes a situation that allows to choose symmetric $\mu_n$ supported in $B_n$ so that the above sequence of inequalities is quite tight
\begin{equation*}
 \delta_\Gamma \cdot n - 2\log n +O(1) \le -2\log (\| \lambda(\mu_n)\|) \le h(\mu_n) \le \delta_\Gamma\cdot \ell(\mu_n) \le  \delta_\Gamma \cdot n
\end{equation*}

In conclusion, we should point out that some of the estimates, although not in the sharpest form can be deduced from the Rapid Decay property (RD), which is known for hyperbolic groups. 

\subsection{Organization of the paper}\hfill{}\par

We set the notation and recall the notion of quasiruled hyperbolic spaces in \S~\ref{background}. \S~\ref{representations} is dedicated to unitary representations of discrete groups used in this paper. Theorems~\ref{mainradius} and~\ref{T:specialspec} are proved in \S~\ref{estimate}. In \S~\ref{approximation}, we deduce Theorems~\ref{maintheorem} and~\ref{maintheoremB}. Finally, \S~\ref{BAD} discusses the proof of Theorem~\ref{wellapproximable}. Some additional remarks are given in \S~\ref{spectraloptimality}.

\subsection{Acknowledgements}\hfill{}\par

The author would like to thank Alex Furman for suggesting the problem and numerous useful conversations. 

\section{Background and Notation}\label{background}

We will use Landau's asymptotic notation: $f(x)=O(g(x))$ means that there exists constant $K>0$, so that $|f(x)| \leq Kg(x)$. For a function $h:X \to \bbR$(where $X$ is a general space), we will write $h=O(1)$ meaning that $h$ is a bounded function.

\subsection{Quasi-ruled hyperbolic spaces}\hfill{}\par

Let $(X,d)$ be a metric space. For $x,y,z \in X$ the Gromov product is defined by  
\begin{equation*}
(x|y)_z := \half \left(d(x,z)+d(y,z)-d(x,y)\right)
\end{equation*}
The notion of hyperbolicity is usually studied in the setting of complete geodesic spaces. In this paper we are interested to exploit the hyperbolicity of non-geodesic metric spaces. For our purposes we want a notion for which the boundary theory and the theory of quasiconformal measures still exist. We recall the theory of quasiruled hyperbolic spaces (see appendix of \cite{Blachere-Haissinsky-Mathieu} for more details)

\begin{defn}
Let $X$ be a proper metric space.
\begin{itemize}

\item[(1)] A $(\lambda,c)$-\textit{quasigeodesic curve} (resp. ray, segment) is the image of $\bbR$ (resp. $\bbR_+$, a compact
interval of $\bbR$) by a $(\lambda,c)$-quasi-isometric embedding. 
\item[(2)] A $\tau$-\textit{quasiruler} is a quasigeodesic $g : \bbR \to X$ (resp. quasisegment $g : I \to X$,
quasiray $g : \bbR+ \to X)$ such that, for any $s < t < u$, we have
\[ (g(s)|g(u))_{g(t)} \leq \tau\]
\item[(3)]
We say that $X$ is \textit{quasi-ruled} if there exist constants $\lambda \geq 1$ and $\tau, c \geq 0$ such that any two points in $X$ can be joined by a $(\lambda, c)$-quasigeodesic, and every $(\lambda, c)$-quasigeodesic is a $\tau$-quasiruler.
\item[(4)] A \textit{quasitriangle} is given by three points $x, y, z \in X$ together with three quasirulers(edges) joining them.
\item[(5)] A quasitriangle is $\delta-$\textit{thin} if any of its edges is in the $\delta$-neighborhood of the union of two other edges.
\item[(6)] A quasiruled metric space $X$ is called \textit{hyperbolic} if it satsifies the Rips condition for some $\delta \geq 0$, i.e. every quasitriangle is $\delta-$thin. 

\end{itemize}

\end{defn}

\begin{example} An important example of quasiruled hyperbolic spaces is the class of convex cocompact subgroups of $\Isom(\mathbf{H}^n)$. They act properly cocompactly by isometries on their convex core in $(\mathbf{H}^n,d_{\mathbf{H}^n})$. Fix $x_0\in \mathbf{H}^n$ a basepoint. Define the left invariant metric on $\Gamma$: for $g,h\in \Gamma$, $d(g,h):=d_{\mathbf{H}^n}(g.x_0,h.x_0)$. This metric is quasi-isometric to the word metric on $\Gamma$, and with respect to this metric, $\Gamma$ is itself a proper quasiruled hyperbolic space. 
\end{example}

One of the useful features of thin triangles is that they admit a centroid. More precisely, given three points $x, y, z$, there is a tripod $T$ and an isometric embedding $f : \{x, y, z\} \to T$ such that the images are the endpoints of $T$. We denote by $C_T$ the center of $T$.
\begin{lemma}\label{centroid}(Tripod lemma)(\cite{Blachere-Haissinsky-Mathieu}Lemma A.3) Let $\Delta$ be a $\delta$-thin quasitriangle with vertices $x, y, z$ in a quasiruled hyperbolic space $X$. There is a $(1, c_0)$-quasiisometry $f_\Delta : \Delta \to T$ , where $T$ is the tripod associated with $x, y, z$ and $c_0$ depends only on the data $(\delta, \lambda, c, \tau)$.
\end{lemma}
We call $f_\Delta^{-1} (C_T)$ a centroid of $\Delta$. Of course, the map $f_\Delta$, and thus the centroid are not unique, but there exists a constant $c_1$ depending on the space only, such that for every quasitriangle $\Delta\subset X$, every 2 centroids of $\Delta$ are at most at distance $c_1$.

\subsection{Visual boundary and Patterson-Sullivan measures}\label{psintro}\hfill{}\par
Geodesic hyperbolic spaces admit a visual boundary and conformal densities on it. In a similar fashion, proper quasiruled hyperbolic metric spaces admit a natural boundary, called the $\textit{visual boundary}$ associated to $(X,d,x_0)$ 
\begin{equation*}
\partial X := \setdef { (x_i)_{i=1}^\infty} { x_i \in X, \lim_{i,j \to \infty} (x_i|x_j)_{x_0} =\infty } /\sim
\end{equation*}
where
\begin{equation*}
(x_i)\sim (y_i) \Leftrightarrow (x_i|y_i)_{x_0} \underset{i\to \infty}{\longrightarrow} \infty
\end{equation*}
The visual boundary is the set of equivalence classes of infinite quasiruler rays, where two rays are equivalent if they are at bounded Hausdorff distance from each other. The boundary $\partial X$ doesn't depend on the choice of basepoint $x_0$.
\par
Similarly to geodesic hyperbolic spaces, in a quasiruled hyperbolic space there exists a quasiruled curve between any two points in the boundary.
\par 
The boundary $\partial X$ may be equipped with the topology, whose basis is given by shadows. 
For $y \in X$ and $C\geq 0$, the \textit{shadow} $O_C(x_0,y)$ is 
\begin{equation*}
O_C(x_0,y) := \setdef{ [(z_i)] \in \partial X}{ \liminf_{j \to \infty } (z_j|y)_{x_0} \geq d(x_0,y) - C }
\end{equation*}
Alternatively, a point $\xi \in \partial X$ belongs to the shadow $O_C(x_0,y)$ if some quasiruler ray from $x_0$ to $\xi$ intersects the closed $C$-ball around $y$. 

\par
Sometimes we would like to think of shadows as subsets of $X$, in this case
\begin{equation*}
\bar{O}_C(x_0,y) := \left\{z\in X | (y|z)_{x_0} \geq d(x_0,y) - C \right\}
\end{equation*}
\medskip
For $z\in X$, the \textit{Busemann function at} $z$, $\beta_z : X \times X  \to \mathbb{R}$ is
\begin{equation*}
\beta_z(x,y):= d(z,x) - d(z,y)
\end{equation*}
For $\xi \in \partial X$, we define Busemann function at $\xi$ by 
\begin{equation*}
\beta_\xi(x,y):= \sup_{z_{t}\to \xi} \limsup_{t\to \infty} \left\{ d(z(t),x) - d(z(t),y) \right\}
\end{equation*}
The above $\sup$ should be taken along all possible quasiruler rays $z(t)$ from $y$ to $\xi$.
\medskip

Recall that for $\Gamma<Isom(X,d)$, with a chosen basepoint $x_0\in X$, the \textit{critical exponent} for $\Gamma$ is given by
\[
\delta_\Gamma := \limsup_{R \to \infty} \frac{ \log \# \setdef{ g\in\Gamma}{ d(g.x_0,x_0) \leq R }}{R}
\]

The $\Gamma$ action on $X$ induces natural action on $\partial X$ and on the space of Busemann functions. 
\begin{equation*}
g.\beta_\xi(x,y) := \beta_{g.\xi}(x,y) = \beta_\xi(g^{-1}.x,g^{-1}.y)
\end{equation*}

The next theorem summarizes the main properties of quasiconformal measures on the boundary of $X$. It was proved by Coornaert in \cite{Coornaert} for geodesic hyperbolic spaces, and by Blachere-Haissinsky-Mathieu in \cite{Blachere-Haissinsky-Mathieu} for proper quasiruled hyperbolic spaces.
\begin{theorem}(\cite{Blachere-Haissinsky-Mathieu}, Theorem 2.3)
Let $\Gamma$ be a finitely generated group acting properly cocompactly by isometries on a pointed proper  quasiruled hyperbolic space $(X,d,x_0)$. For any small enough $\epsilon>0$  

\begin{itemize}
\item[(1)] There exists a visual metric $d_\epsilon$ on the boundary $\partial X$, its Haussdorff dimension is given by $ \dim_H(\partial X,d_\epsilon) = \delta_\Gamma/\epsilon $ 
\item[(2)] There exists a $\Gamma-$equivariant family $\{\rho_x\}_{x \in X}$ of Radon probability measures on $\partial X$, i.e. for any $g\in \Gamma, x\in X$ we have $g_* \rho_x = \rho_{g.x}$. Moreover, the entire family $\rho_x$ is in the same measure class.
\item[(3)] The distortion of a measure by the $\Gamma$ action is measured by the Busemann functions, namely for any $\xi \in \partial X$
\begin{equation*}
\frac{d\rho_y}{d\rho_x}(\xi) = e^{-\delta_\Gamma\beta_\xi(y,x)+O(1)}
\end{equation*}
\item[(4)] $\rho_x$ are Ahlfors-regular of dimension $\delta_\Gamma/\epsilon$, i.e. for any $\xi \in \partial X $, for any $r \in(0, diam_{\epsilon} (\partial X))$, we have
 \[\rho_x(B_{d_\epsilon}(a, r)) =  r^{\delta_\Gamma/\epsilon +O(1)}  \]
\item[(5)] $\Gamma$ action on $(\partial X, \rho_x)$ is ergodic for any $x\in X$
\end{itemize}
\end{theorem}

This class of measures is called the \textit{Patterson-Sullivan measure class}. It does not depend on the choice of $\epsilon$. Denote $\rho:= \rho_{x_0}$.
\par 
In fact, the metric $d_\epsilon$ is given in the following way. First one extends the Gromov product to the boundary by defining
\begin{equation*}
([x_i]|[y_i])_{x_0}:= \limsup_{i\to \infty}(x_i|y_i)_{x_0}
\end{equation*}
where  $\limsup$ is taken over all quasiruled rays in the equivaence classes.
There exists $\epsilon_0>0$, such that for any $0<\epsilon<\epsilon_0$ there exists a metric on $\partial X$ satisfying
\begin{equation*}
d_\epsilon(\xi,\eta) := O(1)e^{-\epsilon(\xi|\eta)_0}
\end{equation*}
Such metric $d_\epsilon$ induces the boundary topology described above. Moreover, the shadows are related to the balls in metric $d_\epsilon$.
\begin{prop}\label{diamofshadows}(\cite{Blachere-Haissinsky-Mathieu}, Proposition 2.1)
There exists $C_0\geq 0$, such that for any $C \geq C_0$ and any $x\in X$
\begin{equation*}
\text{diam}_\epsilon (O_C(x_0,x)) =  e^{-\epsilon d(x,x_0) + O(1) }
\end{equation*}
\end{prop}

Combining the fact that Patterson-Sullivan measures are Ahlfors regular with respect to this metric and the description of shadows we can conlude the following corollary known as the lemma of the shadow, 
\begin{cor}\label{rhoofshadows}(Lemma of the shadow, \cite{Blachere-Haissinsky-Mathieu}, Lemma 2.4)
There exists $C\geq 0$, such that for any $x \in X$  
\begin{equation*}
\rho(O_C(x_0,x)) = e^{-\delta_\Gamma d(x,x_0) + O(1)}
\end{equation*}
\end{cor}

The $\Gamma$ action on $(X,d)$ induces the left invariant metric $d_0:=d(g.x_0,h.x_0)$. If the action is proper and cocompact, $(\Gamma, d_0)$ is itself a proper quasiruled hyperbolic space. We denote by $B_n$ the $n-$ball in $\Gamma$ with respect to $d_0$ and define the $k-$shell: 
\[
S_{n,k} := B_n \setminus B_{n-k}
\]
The shadows of the shells $S_{n,k}$ cover the boundary with finitely many overlaps (with the bound uniform in $n$). More precisely, 

\begin{lemma} \label{coveringshadows}(\cite{Coornaert}, Lemma 6.5)
There exist $C,k \geq 0$ such that for any $n \in \mathbb{N}$ 
\[ \underset{g\in S_{n,k}}{\bigcup} O_C(e,g) \supseteq \partial \Gamma \] 
Moreover, there exists $L$(depending only on $C$ and $k$) such that for any $n$ and any $\xi \in \partial G$ 
\begin{equation*}
\# \setdef { g \in S_{n,k} }{ \xi \in O_C(e,g) }  \leq L
\end{equation*}
i.e. every $\xi \in \partial \Gamma$ is covered by at most $L$ shadows of elements in the shell $S_{n,k}$
\end{lemma}

We also have precise asymptotics of the growth of balls and shells

\begin{lemma}\label{sizeofballs}(\cite{Coornaert}, Theorem 7.2) 
There exists $k>0$, such that 
\begin{itemize}
\item[(1)] 
$ \# S_{n,k} = e^{\delta_\Gamma n +O(1)} $

\item[(2)]	
$
\# B_n= e^{\delta_\Gamma n+O(1) }
$	
\end{itemize}
\end{lemma}

Two above lemmas are stated for geodesic hyperbolic spaces in \cite{Coornaert}, but the same proofs will work for quasiruled hyperbolic spaces. 
\begin{defn}\label{defshell}
Fix $k>0$ for which Lemmas~\ref{coveringshadows} and~\ref{sizeofballs} hold. Denote the \textit{shell} $S_n:=S_{n,k}$.
\end{defn}

\begin{defn}\label{ogdef}
Let $\Gamma$ as above. Let $C \geq 0$ be large enough to satisfy Corollary~\ref{rhoofshadows} and Lemma~\ref{coveringshadows}. 
For $g\in \Gamma$, the $g-$\textit{shadow in } $\Gamma$ is a subset of $\partial \Gamma$ given by
\begin{equation*}
O(g):= O_C(e,g)
\end{equation*}
\end{defn}

\section{Some Unitary Representations of $\Gamma$}\label{representations}

A discrete group $\Gamma$ acts on itself by left multiplication which induces the \textit{left regular representation} $\lambda_\Gamma: \Gamma \to U(l^2(\Gamma))$ given by:
\[
\lambda_\Gamma (g) f(h) = f(g^{-1}h) \quad  \text{for } f\in l^2(\Gamma), g\in \Gamma
\]
If $\Gamma$ acts by measure preserving transformations on a probability space $(X, m)$ we can associate with the action the \textit{Koopman representation} $\pi: \Gamma \to U(L^2(X,m))$, which is given by
\[
\pi(g) f (x) = f(g^{-1}.x)\quad  \text{for } f\in L^2(X), g\in \Gamma
\]
The constant functions are invariant, hence we denote by $\pi_0$ the restriction of $\pi$ to the orthogonal complement of the constant functions $L_0^2(X,m)=\setdef{ f\in L^2(X,m)}{ \int_X f dm=0 }$. \par
If, however, the action only preserves the measure class, we can modify the Koopman representaion to become a unitary representation $\pi_X: G \to U(L^2(X, \nu))$:
\[
\pi_X(g) f(x) = f(g^{-1}.x) \sqrt{\frac{dg_* \nu}{d\nu}(x)}
\]
Such $\pi_X$ is called the \textit{quasi-regular representation}. 
\par 
For example, if $\Gamma$ is as in \S~\ref{psintro}, $\Gamma$ acts on its visual boundary equipped with Patterson Sullivan measure. We call the associated quasi-regular representation the \textit{boundary representation} and denote it by $\pi_{\partial \Gamma}$.

Given finitely supported probability measure $\mu$ on $\Gamma$ and a unitary representation $\sigma:\Gamma \to U(H)$ we can average the representation to get a Markov operator $\sigma(\mu):H \to H$ by
\[
\sigma(\mu) = \sum_{g\in \Gamma} \mu(g) \sigma(g)
\]
\begin{example} 
$\lambda_\Gamma(\mu)$ is the Markov operator associated with the random walk on $\Gamma$ with law $\mu$. It is known that $\| \lambda_\Gamma(\mu)\| <1$ if and only if $\Gamma$ is amenable.
\end{example}
\begin{example}
	Let $H<\Gamma$ a subgroup. $\Gamma$ acts on $\Gamma/H$ by left multiplication, which induces the representation $\pi_{\Gamma/H} : \Gamma \to U(l^2(\Gamma/H))$. 
\end{example}

\begin{theorem}(Kesten, \cite{Kesten})
Let  $\mu$ be a uniform measure on some generating set $S$ of $\Gamma$. If $H$ is amenable, then
\[ \| \lambda_{\Gamma}(\mu) \| = \|  \pi_{\Gamma/H}(\mu)\|  \]
\end{theorem}

A generalized version of this is the following:
\begin{theorem} \label{thmrep}(Kuhn, \cite{Kuhn})
Let $\Gamma$ be a discrete group, $\mu \in Prob(\Gamma)$, and let $\Gamma$ act ergodically preserving the measure class on a probability space $(X,\nu)$. Assume the action is amenable in the sense of Zimmer, and let $\pi_X$ the corresponding quasi-regular representation. Then,
\begin{equation*}
\|\lambda_\Gamma(\mu)\| \geq \|\pi_X(\mu) \| 
\end{equation*} 
\end{theorem}

This lemma by Shalom gives a useful condition for an opposite inequality
\begin{lemma}(\cite{Shalom}, Lemma 2.3)\label{shalomrep}
Let $\pi$ be a unitary $\Gamma$-representation, with a positive $\Gamma$-vector, that is nonzero vector $v\in \mathcal{H}$, such that $\left< \pi(g)v,v \right> \geq 0$ for all $g\in \Gamma$. Then for any finitely supported probability measure $\mu$ on $\Gamma$
\begin{equation*}
\|\lambda_\Gamma(\mu)\| \leq \|\pi_X(\mu) \| 
\end{equation*} 
\end{lemma}

\begin{example}\label{exampleamenaction} 
An example of an amenable action is the action of convex cocompact subgroup $ \Gamma<\SL_2(\bbR)$ on its Poisson boundary (which can be identified with the visual boundary $\partial \Gamma$) equipped with Patterson Sullivan measure(\cite{Zimmer}). Moreover, $\pi_{\partial \Gamma}$ has a positive $\Gamma$-vector(e.g. a constant function), thus we can deduce that for any probability measure $\mu$ on $\Gamma$ we have 
\[ \|\lambda_\Gamma(\mu)\| = \|\pi_{\partial \Gamma}(\mu)\|  \] 

\end{example}

The following theorem is folklore. It relates the left regular representation and the Koopman representation on the two torus.
\begin{theorem}\label{toruskoopman}
Let $\Gamma<SL_2(\mathbb{Z})$ act on the torus $\mathbb{T}^2$ equipped with Lebesgue measure $m$, $\pi_0$ be the Koopman representation on $L^2_0(\mathbb{T}^2)$. Then, for any probability measure on $\Gamma$
 \begin{equation*}
  \| \pi_0(\mu) \| = \|\lambda_\Gamma(\mu) \| 
 \end{equation*}
\end{theorem}
\begin{proof}
Recall that the Fourier transform is an isometry between
\[
\widehat{ } \quad : L_0^2(\mathbb{T}^2) \to \ell^2(\mathbb{Z}^2 \setminus 0)
\]
defined as following: for $f\in L_0^2(\mathbb{T}^2)$
\[  
\widehat{f}(\vec{n}) = \int_{\mathbb{T}^2} f(x) e^{2\pi i \left< \vec{n}, x \right>}dm(x)
\]
$\Gamma$ acts on $\mathbb{Z}^2 \setminus 0$ via left multiplication by transpose matrix. This induces a representation $\widehat{\pi_0}$ on $\ell^2(\mathbb{Z}^2 \setminus 0)$ given by
\[
\widehat{\pi_0}(g) \widehat{f}(\vec{n}) = \widehat{f}\left( g^{T} \vec{n}\right)
\]
The following  diagram commutes
\[
\begin{tikzcd}
L^2_0(\mathbb{T}^2) \arrow{r}{\widehat{} } \arrow[swap]{d}{\pi_0(g)} & L^2_0(\mathbb{Z}^2\setminus 0) \arrow{d}{\widehat{\pi_0(g)}} \\
L^2_0(\mathbb{T}^2) \arrow{r}{\widehat{} }  & L^2_0(\mathbb{Z}^2\setminus 0)
\end{tikzcd}
\]
The Fourier transform intertwines the representations. Hence, $\|\pi_0(\mu)\|=\| \widehat{\pi_0(\mu)}\|$.

Pick representatives from each $\Gamma-$orbit of $\widehat{\pi_0}$: $D=\{v_1,v_2,v_3,...\}$. Then, 
\[\mathbb{Z}^2 \setminus 0 \cong \underset{i}{\bigcup} \Gamma/Stab(v_i)\]
and 
\[
\widehat{\pi_0} = \underset{i}{\bigoplus} \pi_{\Gamma/Stab(v_i)}
\]
hence, $\|\widehat{\pi_0}(\mu)\| =\underset{i}{\sup} \| \pi_{\Gamma/Stab_{v_i}}(\mu)\|$.
The stabilizers of vectors in $\mathbb{Z}^2\setminus 0$ are amenable (conjugate to the group of upper triangular matrices), thus by Kesten's theorem we have $\| \pi_{\Gamma/Stab(v_i)}(\mu)\| = \| \lambda_{\Gamma}(\mu) \|$ for every $i$, and hence $\| \pi_0(\mu) \|= \| \widehat{\pi_0}(\mu) \|=  \| \lambda_\Gamma(\mu) \|$ \qedhere
\end{proof}
Combining results from this section we have
\begin{cor}\label{koopmanbdry}
Let $\Gamma<SL_2(\mathbb{Z})$  convex cocompact. Let  $\lambda$ be the left regular representation, $\pi_{\partial \Gamma}$ the boundary representation as described in \S~\ref{psintro} and $\pi_0$ the Koopman representation on the torus. Let $\mu \in Prob(\Gamma)$ finitely supported measure, such that the support generates the entire group, then
\begin{equation*}
\| \pi_{\partial \Gamma}(\mu) \| = \| \pi_0(\mu) \| = \| \lambda(\mu) \|
\end{equation*}
\end{cor}

\section{The Spectral Estimate for the Boundary Representation}\label{estimate}
In this section we prove Theorem~\ref{mainradius}. Consider a group $\Gamma$ that acts by isometries properly cocompactly on a proper quasiruled hyperbolic space $(X,d)$. Fix $x_0 \in X$ a basepoint. We will abuse the notation and use $d$ as a metric on a group, i.e. $d(g,h):= d(g.x_0,h.x_0)$. With this metric, $(\Gamma,d)$ is a proper quasiruled hyperbolic space. Let $\delta_\Gamma$ be the critical exponent of $\Gamma$. For every  $n \in \mathbb{N}$, let $\mu_n$ be a uniform probability measure on the shell $S_n$ (as defined in~\ref{defshell}). 

By Corollary~\ref{koopmanbdry}, Theorem~\ref{T:specialspec} and Theorem~\ref{mainradius} follow immediately from the theorem below.

\begin{theorem} \label{specradthm}\hfill{}
Let $(\Gamma,d)$ and $\mu_n$ as above. Let $\rho$ be Patterson-Sullivan measure on $\partial \Gamma$ and $\pi_{\partial \Gamma}:\Gamma \to U(L^2(\partial \Gamma), \rho)$ the corresponding quasiregular representation of $\Gamma$ on the boundary. Then 
\begin{equation*}
\| \pi_{\partial \Gamma}(\mu_n) \| \leq  e^{-\half \delta_\Gamma n + 2 \log n + O(1) }
\end{equation*}
\end{theorem}

We will call $\pi_{\partial \Gamma}(\mu_n)$ the boundary operators. We fix $n$ throughout the proof. We will bound the operator norm of the boundary operator by testing it on a dense set of simple functions. For each $r\in \mathbb{N}$ we will construct a finite dimensional operator $\Pi_r$  that mimics the application of $\pi_{\partial \Gamma}(\mu_n)$ to a step function $f$ (for large enough $r$ that depends on the complexity of $f$). We will then study $\Pi_r$ and relate their operator norms to the operator norm of $\pi_{\partial \Gamma}(\mu_n)$. 
\par
Let $r \in \bbN$. Enumerate the elements $\{ g_j \}$ in the shell $S_r \subset \Gamma$. Denote by $O_j=O(g_j)$ the shadows as defined in~\ref{ogdef}, and their characteristic functions by $\chi_j=\chi_{O_j}$. Define a $|S_r|\times |S_r|$ matrix $\Pi_r(\mu_n)$ by
\begin{equation*}
\left( \Pi_r(\mu_n) \right)_{ij} := \left< \pi_{\partial \Gamma}(\mu_n)\chi_i,\chi_j \right>  = \int_{\partial \Gamma} \left(\pi_{\partial \Gamma}(\mu_n) \chi_i\right)(\xi) \chi_j(\xi) d\rho(\xi)
\end{equation*}
The main step will be estimating the operator norms of finite dimensional operators $\Pi_r(\mu_n)$
\begin{theorem}\label{matrixnorm}
For $\Pi_r(\mu_n)$ as above we have
\begin{equation*}
 \| \Pi_r(\mu_n)\| \leq e^{-\delta_\Gamma r  - \half \delta_\Gamma n + 2\log n + O(1) }
\end{equation*}

\end{theorem}

In \S~\ref{reductionlinalg} we will show that Theorem~\ref{matrixnorm} implies Theorem~\ref{specradthm}. In \S~\ref{matrixnormsec} we will prove Theorem~\ref{matrixnorm}

\subsection{Reduction to linear algebra}\label{reductionlinalg}\hfill{}\par
\begin{proof} (Theorem~\ref{matrixnorm} $\Longrightarrow$ Theorem~\ref{specradthm}) 

Recall that 
\[
\| \pi_{\partial \Gamma}(\mu_n) \| = \sup_{\|f\|=1} \left< \pi_{\partial \Gamma}(\mu_n) f , f \right>
\]
Since $\pi_{\partial \Gamma}(\mu_n)$ is an operator preserving the cone of positive functions, it is sufficient to take the supremum only over non-negative functions(or a dense subset of it). 
 \par 
We fix a visual metric $d_\epsilon$ for some small enough $\epsilon>0$. Recall that the balls in the visual metric generate the topology. We consider
\[ H_+ :=\setdef { f=\sum_{i=1}^{t} a_i \chi_{I_i} }{ a_i> 0 ,  I_i \subseteq \partial \Gamma \text{ disjoint closed balls},\|f\|=1 }
 \]
$H_+$ is clearly dense in the set of non-negative functions of norm 1. 
\medskip

Our strategy will be to show that for each $f \in H_+$ there exists $r>0$ and a vector $\vec{v}\in \mathbb{R}^{|S_r|}$ such that 
\begin{itemize}
\item[(M1)] $\left< \pi_{\partial \Gamma}(\mu_n)f,f \right> \leq \vec{v}^T \Pi_r(\mu_n) \vec{v}$
\item[(M2)] $\|\vec{v}\|^2 \leq e^{\delta_\Gamma r + O(1)}  $
\end{itemize}
where $\|\vec{v}\|$ is the Euclidean norm on $\mathbb{R}^{|S_r|}$. \par

This, combining with Theorem~\ref{matrixnorm} will imply that for each $f\in H_+$ we have some $\vec{v}$ and $r$ satisfying
\begin{align*}
\left< \pi_{\partial \Gamma}(\mu_n)f,f\right> &\le e^{\delta_\Gamma r + O(1)}\frac{\vec{v}^T \Pi_r(\mu_n) \vec{v}}{\|\vec{v}\|^2} 
&\le e^{\delta_\Gamma r +O(1)} \| \Pi_r(\mu_n) \| &\le \\ &\le e^{-\half \delta_\Gamma n + 2\log n +O(1)}
\end{align*}
Taking the supremum over $f\in H_+$ will finish the proof of Theorem~\ref{specradthm}.
\medskip

We are left to construct $v$ from $f$ satisfying the properties (M1) and (M2). Fix an element in $H_+$ of the form $f=\sum_{i=1}^{t} a_i \chi_{I_i}$ with $\| f\|=1$ .
Denote by $I_{i+\eta}$ the closed balls having the same centers as $I_i$, but with radius larger by $\eta$. Fix $\eta>0$ such that for every $1 \leq  i\leq t$ we have $\rho(I_{i+\eta}) \leq 2\rho(I_i)$ and so that $I_{i+\eta}$ are pairwise disjoint for all $i$. Such $\eta$ exists, since $I_i$ is a finite family. By Proposition~\ref{diamofshadows} bounding the diameter of the shadows we can find $r$ large enough, so that two following conditions are satisfied:
\begin{itemize}
\item[(S1)] for every $g_j \in S_r$ we have $diam (O_j) \leq \frac{1}{3}\min_{i,i'} d_\epsilon (I_i,I_{i'})$
\item[(S2)] for every $g_j \in S_r$ we have $diam (O_j) \leq \eta$
\end{itemize}
\par
For each $1\leq j\leq |S_r|$ define

\[v_j =
\left\{
	\begin{array}{ll}
		a_i  & \mbox{if } \exists i \text{ s.t. } I_i\cap O_j \neq \emptyset  \\
		0 & \mbox{otherwise}
	\end{array}
\right.
\]
$\vec{v}=(v_j)$ is well defined since by condition (S1) each $O_j$ intersects at most one of the sets from the family $\{I_i\}$. 
\[ \]

Let $f^v = \sum_{j=1}^{|S_r|}v_j \chi_j$ 

By Theorem~\ref{coveringshadows} there exists $L\in \mathbb{N}$ so that each point in the boundary is covered by at most $L$ different shadows of elements in $S_r$. Combining it with (S2)  we have for each $1\leq i\leq t$
\begin{equation}\label{covereqn}
\chi_{I_i} \leq \sum_{j:O_j \cap I_i \neq \emptyset} \chi_j \leq  L \chi_{I+\eta} 
\end{equation}
In particular from the left inequality in~\eqref{covereqn}
\begin{equation*}\label{fvandfw0}
f \leq f^v
\end{equation*}

It follows now that $\vec{v}$ satisfies (M1), i.e.
\begin{eqnarray*}
\left< \pi_{\partial \Gamma}(\mu_n)f,f \right> &\le \left< \pi_{\partial \Gamma}(\mu_n)f^v,f^v \right> &=  \vec{v}^T \Pi_r(\mu_n) \vec{v}
\end{eqnarray*}
To show (M2) we are left to estimate the of $\vec{v}$
\begin{align*}
\| \vec{v} \| ^2 = \sum_i \sum_{j:O_j \cap I_i \neq \emptyset} a_i^2 &\stackrel{(1)}{=} e^{\delta_\Gamma r +O(1)}\sum_i a_i^2 \sum_{j:O_j \cap I_i \neq \emptyset} \rho(O_j) 
\stackrel{(2)}{\le} \\
&\stackrel{(2)}{\le}  e^{\delta_\Gamma r +O(1)} \sum_i a_i^2 L \rho(I_{i+\eta}) 
  \stackrel{(3)}{\leq} \\
    &\stackrel{(3)}{\leq} e^{\delta_\Gamma r +O(1)}2L  \sum_i a_i^2 \rho(I_{i}) 
\le \\ &\le e^{\delta_\Gamma r }\|f\|^2 =  e^{\delta_\Gamma r +O(1) }
\end{align*}
The first equality follows from Corollary~\ref{rhoofshadows}, which, if applied here, states $\rho(O_j)=e^{-\delta_\Gamma r +O(1) }$, the second follows from integrating the right  inequality in~\ref{covereqn}, the third is obtained from our choice of $\eta$(since $\rho(I_{i+\eta})\leq 2\rho(I_i)$). This finishes the proof.
\end{proof}

\subsection{Hyperbolic geometry}\hfill{}\par
The proof of Theorem~\ref{matrixnorm} relies on the hyperbolicity of the metric $d$ on $\Gamma$. We prove a sequence of technical lemmas which will be necessary in \S~\ref{matrixnormsec}. 

\begin{lemma}\label{busemannapprox}
There exist  $R, \Delta \geq 0$ depending only on $(\Gamma,d)$, such that for any $r>n+\Delta$, for any $g \in S_r$, $\xi \in O(g)=O_C(g,e)$ and $h\in S_n$ we have 
\begin{equation}
| \beta_\xi(h,e) - \beta_{g}(h,e) | \leq R
\end{equation}
\end{lemma}

\begin{proof}
Let $\Delta$ be the maximal thickness of quasi-triangles in $\Gamma$. We will show that $R=4(\tau +C+\Delta)+1$ suffices. Let $r>n+\Delta$ and choose $g\in S_r$ and $\xi \in O(g)$. 
\par 
Let $z(t)$ be a quasiruler from $e$ to $\xi$ s.t. for some large $t_0$ we have
$ \beta_\xi(h,0) - d(z(t_0),h)-d(z(t_0),e)  \leq 1$. Note that by definition of $O(g)$ there is some quasiruler from $e$ to $\xi$, that passes in a $C-$neighborhood of $g$. Using $\Delta$-thinness of triangles, we can conclude that any quasiruler from $e$ to $\xi$ must pass in a $(C+\Delta)-$neighborhood of $g$. Let $s \in \mathbb{R}$, so that $z(s)$ is at distance at most $C+\Delta$ from $g$. We can assume $t_0 > s$. Then, 
\[ 0 \leq   d(e,z(s)) + d(z(s),z(t_0)) - d(e,z(t_0))  \leq 2\tau \]
The right hand side of the above inequality holds since $z(t)$ is a $\tau$-quasiruler, and the left hand side is the triangle inequality.
\par 
Let $z'(t)$ be a quasiruler between $h,z(t_0)$. Similarly, $z'(t)$ has to pass through the $C+\Delta$-neighborhood of $g$. Let $s'$ such that $z'(s')$ is in the $C+\Delta$ neighborhood of $g$. Similarly, by the property of quasiruler for $z'(t)$
\[ 0 \leq  d(h,z'(s')) + d(z'(s'),z(t_0))  -d(h,z(t_0)) \leq 2\tau \]
Noting that $z'(s')$ and $z(s)$ are $(C+\Delta)-$close to $g$ and $z(t_0)=z'(t'_0)$, we can substract two of the above inequalities to get 

\[| \beta_\xi(h,e) - \beta_g(h,e) | \leq   4\tau +1+4(C+\Delta)=R \]
\end{proof}

\begin{cor} \label{rndbound}
With $\Delta$ as in Lemma~\ref{busemannapprox}, for any $r>n+\Delta$ and for each $g\in S_r, h\in S_n, \xi \in O(g)$ we have
\begin{equation*}
\frac{dh_* \rho}{d \rho}(\xi) =e^{-\beta_{g}(h,e)\delta_\Gamma +O(1)}
\end{equation*}
\end{cor}
Define
\begin{equation}\label{xadef}
X_a(g,n) = \setdef{ h \in S_n }{ n-2a-R < -\beta_{g}(h,e) \leq n-2a } 
\end{equation}
where $R$ is the constant from the Lemma~\ref{busemannapprox}. Without loss of generality we can take $R$ enough large, so that our estimate for the size of the shells $S_{n,R}$ from Lemma~\ref{sizeofballs} holds.  

\begin{lemma} \label{sizeofxa}
With $n,r$ as above, for any $0 \leq a \leq n$ and $g\in S_r$ we have
\begin{equation*}
\# X_a(g,n)    \leq e^{\delta_\Gamma a +O(1)} 
\end{equation*}
\end{lemma}

\begin{proof}
Let $g\in S_r$. Fix a quasiruler between $e,g$. Given $h\in X_a(g,n)$ complete it to the quasitriangle $e,g,h$. By Lemma~\ref{centroid} it is $(1,c_0)$-quasiisometric to a tripod (with $c_0$ depending only on the global quasiruled hyperbolic structure). Hence, the following equations carry on to the tripod via the quaiisometry
\begin{align*} 
	d(e,g) &= r+O(1) \\ 
	d(e,h) &= n+O(1) \\
	d(g,e) - d(g,h)	 &= n-2a+O(1)
\end{align*} 

Let $y(h)\in \Gamma$ be some preimage of the closest point to the centroid of the tripod(if it is not unique, we can choose one). Solving in the tripod it is easy to see that $d(y(h),e)=n-a+O(1)$ for every $h\in X_a(g,n)$. This will ensure that the location of the centroid $y=y(h)$ doesn't depend on which quasitriangle we chose (up to bounded distance), i.e. it doesn't depend on $h\in X_a(g,n)$. Also, $d(y,h)=a+O(1)$ for every $h\in X_a(g,n)$, hence implying that $X_a(g,n)\subseteq B(y,a+O(1))$. By Lemma~\ref{sizeofballs} we can estimate
\[ \# X_a(g,n) \leq \# B(y,a+O(1)) \leq  e^{\delta_\Gamma a +O(1)}  \]

\qedhere

\end{proof}

\begin{lemma}\label{rhoshadows}
With $n,r$ as above, enumerate the elements of $S_r=\{ g_1,g_2,...,g_{|S_r|}\}$. Then for any $i\in \{1,...,\#S_r\}$ and for any $h \in S_n$ we have 
\begin{equation}
\sum_j^{\#S_r}   \rho(O_i \cap hO_j) \leq  e^{-\delta_\Gamma r+\log n +O(1)} \label{eq3}
\end{equation}
\end{lemma}

\begin{proof}
Given $h \in S_n, g_i \in S_r$, we first count $g_j\in S_r$ for which $O_i \cap h O_j \neq \emptyset$ and then we can estimate the measures of the intersections.
\par 
Let $g_j \in S_r$ such. There are two different cases:

\textbf{Case 1}: $d(hg_j,e) \leq d(g_i,e)$. 
\par 
In order for the intersection to be nontrivial, $hg_j$ must lie within distance $2C+\Delta$ from a quasiruler segment $[e , g_i]$. Also, since $h\in S_n$, $hg_j$ must be at most within distance $n+O(1)$ from $g_i$. 

The number of elements in $\Gamma$ lying in a bounded distance from some quasiruler $[e,g_i]$ and being distance at most $n+O(1)$ from $g_i$ is at most $O(n)$. Thus there are at most $O(n)$ possible $g_j$ satisfying $O_i \cap h O_j \neq \emptyset$. In this case $hO_j\cap O_i \subseteq O_i$. The contribution of $\rho(hO_j\cap O_i) = \rho(O_i)$ for each such $j$ to the sum in the equation~\eqref{eq3} is $e^{-\delta_\Gamma r +O(1)}$ by Lemma~\ref{rhoofshadows}, hence the total contribution is at most $e^{-\delta_\Gamma r+\log n +O(1) }$.

\textbf{Case 2}:  $d(hg_j,e) \geq d(g_i,e)$. 
\par 

In this case $hg_j$ should lie in the shadow of $g_i$. Also, $hg_j$ can be at distance at most $n+O(1)$ from $g_i$. In this situation we have $O_i \cap hO_j \subset hO_j$, and in particular $\rho(O_i \cap hO_j) \leq \rho(hO_j) e^{-\delta_\Gamma r - d(hg_j,g_i)+O(1)}$  

For each integer $0\leq b \leq n+O(1)$, there are at most $e^{\delta_\Gamma b+O(1)}$ possible elements $g_j$ for which we can have $ b-1\leq d(hg_j,g_i) \leq b$, and for each such $g_j$ we have $\rho(hO_j) = e^{-\delta_\Gamma r - \delta_\Gamma b +O(1)}$. Thus for fixed $b$ the total contribution of those elements to the sum in~\eqref{eq3} is at most $e^{-\delta_\Gamma r +O(1)}$, and summing over $0\leq b\leq n+O(1)$ we get the contribution of $e^{-\delta_\Gamma r +\log n +O(1)}$

Adding both cases yields the desired estimate.
\end{proof}

\subsection{Estimating the operator norms of $\| \Pi_r(\mu_n)\|$} \label{matrixnormsec}\hfill{}\par
In this section we prove Theorem~\ref{matrixnorm}. We will use an old fact known as Gershgorin circle theorem. It states that the spectral radius of a matrix is bounded by the maximum of the $\ell_1$-norms of the columns.

\begin{proof} (of Theorem~\ref{matrixnorm})
By Gershgorin circle theorem it is sufficient to show that sum of every column in $\Pi_r(\mu_n)$ is bounded by $ e^{-\delta_\Gamma r -\half \delta_\Gamma n +\log n +O(1) }$. 

Recall, in~\eqref{xadef} we defined
\begin{equation*}
X_a(g_i,n) = \setdef { h \in S_n }{ n-2a-R \leq -\beta_{g_i}(h,e) \leq  n-2a } 
\end{equation*}
Note that for any fixed $i$ we have
\[ 
S_n = \bigcup_{a=0}^{n} X_a(g_i,n)
\]
We now evaluate the sum of $i-$th column

\begin{align*}
\sum_{j=1}^{\#S_r} \left< \pi_{\partial \Gamma}(\mu_n) \chi_j, \chi_i \right> &= \sum_{h\in S_n} \mu_n(h) \sum_{j=1}^{\#S_r} \left< \pi_{\partial \Gamma}(h) \chi_j, \chi_i \right> \leq 
\end{align*}
\begin{align}
&\leq \sum_{a=0}^{n} \sum_{h \in X_a(g_i,n)} \mu_n(h) \sum_{j=1}^{\#S_r} \int_{\partial \Gamma} \sqrt{\frac{dh_* \rho}{d\rho}(\xi)} \chi_j(h^{-1}\xi)\chi_i(\xi) d\rho(\xi) \label{eq1}
\end{align}
where 
$\mu_n$ is uniformly distributed on $S_n$, hence by Lemma~\ref{sizeofballs} $\mu_n(h) = e^{-\delta_\Gamma n +O(1)}$. Using Corollary~\ref{rndbound} for the Radon Nykodim derivative we continue~\ref{eq1}
\begin{align}
&\leq O(1) \sum_{a=0}^{n} e^{-\delta_\Gamma n} e^{\half \delta_\Gamma (n-2a)} \sum_{h \in X_a(g_i,n)} \sum_{j=1}^{\#S_r}  \rho(O_i \cap hO_j) \leq \label{eq2}
\end{align}
We use the upper bound for the innermost sum from Lemma~\ref{rhoshadows}, and the size of $X_a(g_i,n)$ from Lemma~\ref{sizeofxa}. Hence , continuing~\eqref{eq2}

\begin{align*}
&\leq O(1) \sum_{a=0}^{n} e^{-\delta_\Gamma n} e^{\half \delta_\Gamma  (n-2a)} e^{\delta_\Gamma a } e^{-\delta_\Gamma r + \log n }
\end{align*}
gathering terms and summing over $a$ we get
\begin{align*}
&\leq  e^{-\half \delta_\Gamma n +2 \log n - \delta_\Gamma r +O(1)}  
\end{align*}
\qedhere
\end{proof}

\section{Diophantine Approximation on the 2-Torus}\label{approximation}

We first prove Theorem~\ref{maintheoremB}. Then we show how to deduce Theorem~\ref{maintheorem} from~\ref{maintheoremB}. 

\subsection{Toral Diophantine approximation for convex cocompact subgroups of $SL_2(\mathbb{Z})$}\hfill{}\par
Consider the natural $SL_2(\mathbb{Z})$ action on the torus $\mathbb{T}^2$, with the Lebesgue measure $m$. Fix a monotonic family $\{\Targ_r\}_{r>0}$ of Lebesgue subsets of measure $m(\Targ_r)=\pi r^2$.  After choosing a basepoint $x_0 \in \mathbf{H}^2$, we get a metric on $\Gamma$ defined by $d(g,h):= d_{\mathbf{H}^2}(g.x_0, h.x_0)$. In this section we prove Theorem~\ref{maintheoremB}

\begin{proof}(of Theorem~\ref{maintheoremB})
The first statement follows from the first Borel Cantelli lemma. Indeed, 
\begin{eqnarray*}
	\sum_{g\in \Gamma} m(g^{-1}\Targ_{\psi(\|g\|)}) &\le& \sum_{n=1}^\infty \sum_{\{g \in \Gamma : e^{n-1} < \|g\|\leq e^n\}} \pi \psi(\|g\|)^2   \\
	&\leq& O(1) \sum_{n=1}^\infty e^{2\delta n} \cdot \psi(e^{n-1})^2 \\
	&=& O(1) \sum_{n=1}^\infty e^{2\delta n} \cdot \psi(e^n)^2  <+\infty
\end{eqnarray*}
The last inequality follows from Lemma~\ref{sizeofballs} giving the upper bound of the cardinality of balls in convex cocompact groups, and the fact that $d(g,e)= 2 \log \|g\|$. The series $\sum_{n=1}^\infty e^{2\delta n} \cdot \psi(e^n)^2$ converges if and only if $\psi$ is as in (1)(by Cauchy condensation test). Therefore, $m$-a.e. $x\in \bbT^2$ belongs to at most finitely many of the sets $g^{-1} \Targ_{\psi(\|g\|)}$, as claimed.
\medskip 

The main point is the second statement. 
Let $\pi$ be the Koopman $\Gamma$-representation on $L^2(\mathbb{T}^2,m)$, and $\pi_0$ the restriction to $L^2_0(\mathbb{T}^2,m)$. Let $\mu_n$ be a sequence of probability measures on $\Gamma$, as given in Theorem~\ref{specradthm}. Observe that 
\[
\max \setdef { \|g\| }{ g\in supp(\mu_{2n}) } \leq e^n
\]

We denote 
\[ C_n:= \Targ_{\psi(e^n)} , \qquad E_n=X\setminus \bigcup_{g\in\Gamma, \|g\|\leq e^n} g^{-1} \Targ_{\psi(e^n)}.\]
$C_n$ represents the targets that we are supposed to hit by applying matrices $g$ with $\|g\| \leq e^n$(or equivalently $d(g,e)\leq 2n$). A point belongs to $E_n$ if and only if none of it's translates by $g$ with $\|g\| \leq e^n$ hits the target $C_n$. Hence, we want to show that
\[
	m(E)=0 \qquad\textrm{where}\qquad E=\limsup_{n\to\infty} E_n.
\]
The projections of characteristic functions of $C_n$ and $E_n$ to $L^2_0(X,m)$ are
\[
	h_n=1_{C_n}-m(C_n),\qquad f_n=1_{E_n}-m(E_n) 
\]
Note that
\[
\| h_n\|_2^2 \leq (1-m(C_n))^2m(C_n) \leq m(C_n)
\]
Thus,
\[
	\| h_n\|_2 \leq m(C_n)^{\frac{1}{2}} = O(1) \psi(e^n).
\]
Similarly,
\[ \|f_n\|_2 \leq m(E_n)^{\frac{1}{2}} \]
For any $g\in \Gamma$   
\[
	\langle\pi_0(g)h_n,f_n\rangle=m(C_n)\cdot m(E_n)-m(g^{-1}C_n\cap E_n).
\]
Since any $g \in supp(\mu_{2n})$ satisfies $\|g\| \leq e^n$, one has $g^{-1}C_n\cap E_n=\emptyset$ and 
\[
	\langle\pi_0(g)h_n,f_n\rangle=m(C_n)\cdot m(E_n)
\] 
and consequently 
\begin{eqnarray*}
	m(C_n)\cdot m(E_n) &=&\langle \pi_0(\mu_{2n})h_n,f_n\rangle 
	\leq \|\pi_0(\mu_{2n})\| \cdot\|h_n\|_2\cdot\|f_n\|_2\\
	&\le& \|\pi_0(\mu_{2n})\| \cdot m(C_n)^{\frac{1}{2}}\cdot m(E_n)^{\frac{1}{2}}.
\end{eqnarray*}

By Corollary~\ref{koopmanbdry} and Theorem~\ref{specradthm} we have 
\begin{equation*}
\| \pi_0(\mu_{2n}) \| \leq  e^{- \delta_\Gamma n+ 2 \log n +O(1) }
\end{equation*}
Therefore
\[
	m(E_n)^{\frac{1}{2}}\leq \|\pi_0(\mu_{2n})\| \cdot m(C_n)^{-\frac{1}{2}}
	\leq e^{-\delta_\Gamma n+ 2 \log n +O(1)}\cdot \psi(e^n)^{-1} \]
Hence,
\begin{equation*}
	\sum_{n=1}^\infty m(E_n)\leq O(1)\sum_{n=1}^\infty n^4 e^{-2\delta_\Gamma n} \cdot \psi(e^n)^{-2} <+\infty.
\end{equation*}
where the convergence of the above series is equivalent to convergence of $\sum_{n=1}^\infty (\log n)^4 n^{-2 \delta_\Gamma -1}\psi(n)^{-2}$(by Cauchy condensation test). Consequently, $m(\limsup E_n)=0$. 
\qedhere
\end{proof}

\begin{remark}\label{strongerapproxexplain} In fact, the statement we proved here is a bit stronger than the one that appears in the theorem. We showed that for Lebesgue a.e. point in the torus $x\notin \limsup E_n$, which means that for some large $N$, 
$x\in E^c_n$ for every $n>N$. In other words, not only we have infinitely many solutions for the problem $g.x\in \Targ_{\psi(\|g\|)}$, but for any $n>N$, we have such a solution $g\in \Gamma$ with $e^{n-k}\leq \|g\| \leq e^{n}$, for some fixed $k$. This justifies Remark~\ref{strongerapprox}. 
\end{remark}

\begin{remark}
	One might formulate a simultaneous approximation problem. Given a $d-$tuple of monotonic target families $\{\Targ_r^1, ..., \Targ_r^d\}$ as before and $x_1,...,x_d \in \mathbb{T}^2$, can one find infinitely many $g\in \Gamma$ with $g.x_i \in \Targ_{\psi(\|g\|)}^i$ for each $1\leq i \leq d$? We remark that if one had sharp spectral estimates for $\| \pi_0^{\otimes d}(\mu_n)\|$, a proof similar to Theorem~\ref{maintheoremB} would provide the rates for which the approximation is possible for a.e. $d-$tuple $(x_1,...,x_d)$. 
\end{remark}

\subsection{Reduction to the convex cocompact space}\hfill{}\par
In this section we prove Theorem~\ref{maintheorem}. The proof of the first statement is the exact replica of the proof of Theorem~\ref{maintheoremB}. Note, that since the group is not convex cocompact, we cannot use Lemma~\ref{sizeofballs} for the precise asymptotics of the growth of balls. However, it is sufficient for the proof to bound the cardinality of the balls of radius $n$ in the group by $ e^{(\delta_\Gamma + \epsilon)n+O(1)}$ for arbitrarily small $\epsilon$, and this is possible from the definition of the critical exponent. 
\medskip

We now show that the second part follows from Theorem~\ref{maintheoremB}. Let $\Gamma< \SL_2(\bbZ)$ arbitrary subgroup. Let $\epsilon >0$. We want to show that there are infinitely many solutions $g\in \Gamma$ to $g.x \in \Targ_{\psi(\|g\|)}$ with $\psi(R)=R^{-\delta_\Gamma+\epsilon}$.  
For $\delta_\Gamma=0$  it is trivial, 
so we will assume that $\Gamma$ is nonelementary.
\medskip

The goal is to construct a convex cocompact subgroup $\Gamma_\epsilon<\Gamma$, so that the $\delta_{\Gamma_\epsilon}>\delta_\Gamma-\epsilon$. Since for large $R$ we have $\psi(R)=R^{-\delta_\Gamma+\epsilon}>  R^{-\delta_{\Gamma_\epsilon}}\log^{2.5 + \epsilon}R$, we can apply Theorem~\ref{maintheoremB} to find infinitely many solutions $g\in \Gamma_\epsilon<\Gamma$ to $g.x \in \Targ_{\psi(\|g\|)}$. This proves Theorem~\ref{maintheorem}.
\medskip 

We are left to describe the construction of $\Gamma_\epsilon$. We are inspired by the example provided by Bourgain and Kontorovich in \cite{BourgainKontorovich}(which they attribute to Sarnak). The following trick gives us a way to get rid of parabolic elements in the group.  

\begin{lemma}\label{noparabolics}(\cite{BourgainKontorovich} Remark 1.7, also follows from \cite{Dalbo} Property 3.14) 	
	Let $G=SL_2(\mathbb{Z})$. Let $G(2)=Ker \{ G \to SL_2(\mathbb{Z}/2\mathbb{Z}) \}$ be the congruence subgroup of $G$. Then the commutator subgroup $G(2)'=[G(2),G(2)]$ does not have parabolic elements. 
\end{lemma}

\begin{prop}
	Given $\Gamma<SL_2(\mathbb{Z})$ and $\epsilon>0$ there exists a convex cocompact subgroup $\Gamma_\epsilon < \Gamma$, with $\delta_{\Gamma_\epsilon}> \delta_\Gamma - \epsilon$
\end{prop}
\begin{proof}
	By Sullivan(\cite{Sullivan}, Corollary 6) we know that 
	\[ \delta_{\Gamma} = \sup \setdef{\delta_H} { H < \Gamma \text{ finitely generated}}  \]
	Hence, we can find $\Gamma_0< \Gamma $ finitely generated subgroup with $\delta_{\Gamma_0}> \delta_\Gamma -\half \epsilon$. For Fuchsian groups being finitely generated is equivalent to being geometrically finite (a group is \textit{geometrically finite} if it admits a finitely sided polygon as a fundamental domain in $\mathbf{H}^2$).  Both $G(2)$ and $\Gamma_0$ are such. Susskind showed in \cite{Susskind} that the intersection of two geometrically finite subgroups of a discrete group in $\Isom(\mathbf{H}^n)$ is geometrically finite itself(in fact in dimension 2 it follows from the work of Greenberg \cite{Greenberg}). Hence, $\Gamma_1=G(2) \cap \Gamma_0$ is geometrically finite. 

\par 
Stadlbauer (\cite{Stadlbauer}, Theorem 6.1) proved that if a Kleinian group $G$ is essentially free(and geometrically finite Fuchsian groups are such, see \cite{Stadlbauer} for definition) and $N \unlhd G$ is a normal subgroup,  then $\delta_N=\delta_G$ if and only if $G/N$ is amenable. We can apply this to $\Gamma_1 = \Ker \left\{ \Gamma_0 \to \SL_2(\bbZ/2\bbZ)\right\}$, and then to the commutator subgroup $\Gamma_1'<\Gamma_1$. Hence, $\delta_{\Gamma_1'} = \delta_{\Gamma_1}= \delta_{\Gamma_0} > \delta_\Gamma - \half \epsilon$. 
\par 

Now we apply Sullivan again, to extract a finitely generated subgroup $\Gamma_\epsilon < \Gamma_1'$ with $\delta_{\Gamma_\epsilon} > \delta_{\Gamma_1'}- \half \epsilon > \delta_\Gamma - \epsilon$. Since $\Gamma_\epsilon < G(2)'$, by Lemma~\ref{noparabolics} it has no parabolic elements. This group is convex cocompact, since in dimension $2$ a subgroup is convex cocompact if and only if it is finitely generated and contains no parabolic elements. 
\end{proof}

\section{Approximation of Specific Points in the Torus}\label{BAD}
Theorems~\ref{maintheorem} and~\ref{maintheoremB} only provide us information on approximation properties of Lebesgue almost every point. In this section we wish to characterize Diophantine properties of specific points. We consider $\Gamma<SL_d(\mathbb{Z})$ (with $d\geq 2$) acting on a $d$-torus $\mathbb{T}^d$. For technical reasons we rather use $\sup$-norm on $\bbT^d$ than the Euclidean one. Clearly, this does not affect the approximation properties. For $y\in \bbT^d,r>0$ we denote by $\Boxx (y,r)$ the ball of radius $\half r$ in the $d-$torus in the $\sup-$norm. Note that $m(\Boxx (y,r))=r^d$
\medskip

Naturally, we can not expect a uniform rate of approximation for all target points and all origin points in the torus. Theorem~\ref{wellapproximable} states that under mild assumptions on the acting group,  for given $M$, we can produce a uniform bound for the approximation rate for all targets and all $M-$Diophantine origins. The proof of Theorem~\ref{wellapproximable} relies on two results. First result controls the Fourier coefficients of the measures obtained from a random walk $\mu$ on the torus. If the initial distribution $\delta_x$ is concentrated on a Diophantine point $x\in \bbT^d$, then the Fourier coefficients of the distribution after $k$ steps have exponential decay in $k$. More precisely,
\begin{theorem}\cite{BFLM}\label{bflm}
	Let $\Gamma<SL_d(\mathbb{Z})$ finitely generated group. satisfying (SI) and (PE). Let $\mu \in Prob(\Gamma)$ finitely supported measure, s.t. the support generates $\Gamma$. Let $x\in \mathbb{T}^d$ be $M-$Diophantine.  
	Let $\nu_k = \mu^{*k}*\delta_x$. Then, there exist $c_2>0$, depending only on $\Gamma$ and $\mu$, and $K_0 \in \bbN$ s.t. for $k>K_0$  we have for any $B\in \mathbb{N}$
	\begin{equation*}
	\max_{b\in \mathbb{Z}^d \setminus 0, 0<\|b\|_\infty <B} |\hat {\nu_k}(b) | \leq Be^{-c_2k/M}
	\end{equation*}
\end{theorem}

\begin{definition}
	Let $\nu$ be a probability measure on $\mathbb{T}^d$ and $m$ be the Lebesgue measure. The \textit{discrepancy of } $\nu$ is 
	\begin{equation*}
	D(\nu):=\sup_{P\in J}\left| \nu(P)-m(P) \right|
	\end{equation*} 
	where $J$ is the set of half-open boxes in $\mathbb{T}^d$
	\begin{equation*}
	J:=\left\{ \prod_{i=1}^d [x_i,y_i) :0 \leq  x_i < y_i \leq 1 \right\}
	\end{equation*}
\end{definition}
The second ingredient of the proof is the Erdos-Turan-Koksma inequality. It relates the discrepancy between the distribution $\nu$ and the Lebesgue measure on the torus to the Fourier coefficients of $\nu$ . 

\begin{theorem}[Erdos-Turan-Koksma inequality]\label{etk}
	Let $\nu$ be an atomic probability measure on $\mathbb{T}^d$ with rational values. Let $B$ be an arbitrary positive integer. Then
	
	\begin{equation*}	D(\nu)\leq	C_d 	\left( \frac{1}{B}+ \sum_{0<\|b\|_{\infty}\leq B}\frac{\left| \hat{\nu}(b) \right| }{r(b)} \right)
	\end{equation*}
	where $C_d$ is some explicit constant depending on the dimension $d$.
	\begin{equation*}
	r(b)=\prod_{i=1}^d\max\{1,|b_i|\}\quad\mbox{for}\quad b=(b_1,\ldots,b_d)\in\mathbb{Z}^d.
	\end{equation*}	
	
\end{theorem}

Now we are ready to prove Theorem~\ref{wellapproximable}

\begin{proof} (of Theorem~\ref{wellapproximable})
	Let $\mu$ be the uniform measure on the finite set of generators of $\Gamma$. Let $\nu_k = \mu^{*k}*\delta_x$. Let $\lambda=\max\{ \log \|g\| : g \in supp(\mu) \} $. We will show that $C_\Gamma<\frac{c_2}{d(d+2)\lambda}$ satisfies the theorem, where $c_2$ is the constant from Theorem~\ref{bflm}. 
	
	By submultiplicativity of matrix norm, for every $k>0$
	\begin{equation*}
	\max \setdef{\|g\| }{ g\in supp (\mu^{*k}) }   \leq e^{\lambda k}
	\end{equation*}
	Assume by contradiction that there exists a point $y\in \bbT^d$ which is not $(\Gamma,\frac{C_\Gamma}{M})$-fast approximable. Then, there exists $K>0$, such that for all $k> K$ we have 
	\begin{equation*}
	\nu_k(\Boxx(y,e^{-\frac{\lambda k C_\Gamma}{M}})) =0
	\end{equation*}
	This gives us a lower bound for the discrepancy of $\nu_k$.
	\begin{equation}\label{discrlowerbound}
	D(\nu_k) \geq m(\Boxx(y,e^{-\frac{\lambda k C_\Gamma}{M}})) = e^{- \frac{\lambda k C_\Gamma d}{M} }
	\end{equation}	
	We will now estimate the upper bound for the discrepancy. By Theorem~\ref{etk},
	for every $B,k \in \mathbb{N}$  we have
	\[D(\nu_k)\leq	C_d \left( \frac{2}{B+1}+ \sum_{0<\|b\|_{\infty}\leq B}\frac{\left| \hat{\nu_k}(b) \right|}{r(b)}  \right) 
	\]
	Using $r(b)\geq 1$  and the bound of the Fourier coefficients from Theorem~\ref{bflm} for  $k $ large enough we have
	\[ D(\nu_k) \leq 	 \frac{2C_d}{B}+ C_d (2B+1)^d \cdot Be^{- \frac{c_2k}{M}} 
	\]
	Thus, combining with the lower bound from~\eqref{discrlowerbound}, we have
	\begin{equation}\label{tocontradict} 
	e^{ - \frac{\lambda k C_\Gamma d}{M} } \leq  \frac{2C_d}{B}+ 2^{2d} C_d B^{d+1} \cdot e^{-\frac{c_2k}{M}} 
	\end{equation}
	
	The inequality~\eqref{tocontradict} must hold for all $B \in \mathbb{N}$ and all $k>\max(K_0,K)$, in particular for $B=B(k)=4C_d'(k) e^{\frac{  \lambda k C_\Gamma d}{M} }$ (where we choose the smallest $C_d'(k) \geq C_d$, such that $B(k)$ is an integer. Note that $C'_d(k) \leq 2 C_d$ for large $k$. Then, inequality~\eqref{tocontradict} becomes
	\begin{equation*}
	e^{- \frac{ \lambda k C_\Gamma d}{M}} \leq  \frac{1}{2}e^{-\frac{ \lambda k C_\Gamma d}{M}}+ 2^{2d} C_d (4C_d'(k))^{d+1}  e^{ \frac{\lambda k C_\Gamma d(d+1) -c_2 k }{M}} 
	\end{equation*}
	Multiplying both sides by $e^{\frac{ \lambda k C_\Gamma d}{M}} $ and using $C_d'(k) \leq 2C_d$ we get
	\begin{equation}\label{lasteqtocontradict} 
	1 \leq  \frac{1}{2}+  2^{5d+3} (C_d)^{d+2}  e^{ \frac{(\lambda C_\Gamma d(d+2)  -c_2) k}{M}} 
	\end{equation}
The assumption $C_\Gamma < \frac{c_2}{d(d+2) \lambda}$ implies that the exponent in the right hand side of inequality~\eqref{lasteqtocontradict} is negative, so the above inequality does not hold for arbitrarily large $k$, which gives us the contradiction.
\end{proof}

\section{Spectral Optimality}\label{spectraloptimality}

\subsection{Fundamental inequalities}\hfill{}\par
Consider symmetric finitely supported random walk $\mu$ on a group $\Gamma$.  Denote by $\lambda_\Gamma(\mu)$ the Markov operator, and $\|\lambda_\Gamma(\mu)\|$  the spectral radius of the random walk. Assume $\Gamma$ has a left invariant metric $d$. Assume that $supp(\mu)\subset B_n$. The following inequalities are well known.

\begin{equation}\label{fundiq}
- 2 \log \| \lambda_\Gamma(\mu)\| \stackrel{\text{(1)}}{\leq} h(\mu) \stackrel{\text{(2)}}{\leq} \delta_\Gamma l(\mu) \stackrel{\text{(3)}}{\leq} \delta_\Gamma n
\end{equation}
Part (1) of above inequality was proved by Avez in \cite{Avez}, part (2) is due to Guivarc'h (known as the fundamental inequality of random walks) and part (3) is immediate since the drift is not greater than the maximal length of the elements in the support of $\mu$. 
\medskip

For groups with property of Rapid Decay these inequalities turn out to be asymptotically sharp.

\subsection{Property of rapid decay}\hfill{}\par
Let $\Gamma$ be a discrete group, and $l$ a length function (i.e. $l:\Gamma \to \mathbb{R}_+$, with $l(e)=0, l(g)=l(g^{-1})$, and $l(gh)\leq l(g)+l(h)$ for any $g,h \in \Gamma$. We say that $\Gamma$ has  \textit{property of Rapid Decay(RD) with respect to }$l$ if there exists a polynomial $P(n)$ such that for any $f$ in the complex group algebra $\mathbb{C}\Gamma$ supported on elements of length shorter than $n$ the following inequality holds:
\[
\|f\|_* \leq P(r) \|f\|_2
\]
where $\|f\|_*$ denotes the operator norm of $f$ acting by left convolution on $l^2(\Gamma)$. 

If $d$ is a left invariant metric on $\Gamma$, one can consider $l(g)=d(g,e)$ as the length function. Property RD was first established for free groups by Haagerup, and later Jollisant and de La Harpe(\cite{RD}, \cite{delaHarpe}) proved it for Gromov hyperbolic groups.
\medskip

	In particular, for fixed $n$, let $f(g)=\frac{1}{\#B_n} \chi_{B_n}(g)$, where $\chi_{B_n}$ is the characteristic function of $B_n$. The convolution by $f$ is the operator $\lambda_\Gamma(\mu_n)$ where $\mu_n$ is the uniform distribution on the ball of radius $n$. For any $\epsilon>0$
	\[
	\|f\|_2^2 = \sum_{g\in B_n}\frac{1}{|B_n|^2} = \frac{1}{|B_n|}\leq e^{-(\delta_\Gamma-\epsilon)n + O(1)}
	\]
	Hence we have,
	\[
\|	\lambda_\Gamma(\mu_n)\| \leq  e^{-\half (\delta-\epsilon)n +O(\log n)}
	\]
	Since $\epsilon$ is arbitrary, we just proved the following:

\begin{prop}
	Assume $\Gamma$ has property RD. Then for any $\epsilon>0$
	\[ \rho_{\lambda_\Gamma}(n) \leq e^{-\half (\delta_\Gamma -\epsilon) n} \]
\end{prop}

Theorem~\ref{mainradius} gives a sharper bound for convex cocompact subgroups of $SL_2(\mathbb{Z})$ (or groups that act by isometries cocompactly on a proper quasiruled hyperbolic space $X$), for such $\Gamma$ we have
\[
\rho_{\lambda_\Gamma}(n) \leq n^2 e^{-\half\delta_\Gamma n +O(1)} 
\] 

\subsection{Optimality of random walks}\hfill{}\par
Much work has been done to achieve equality in (2) of the inequality~\eqref{fundiq}. It was shown in \cite{Gouezel} that in hyperbolic case one cannot achieve the equality with a finitely supported measure, unless the group is virtually free. However, one can ask if approaching the equality asymptotically is possible.  
\medskip
By this we mean finding a sequence of finitely supported measures $\mu_n$, so that $\frac{h(\mu_n)}{l(\mu_n)}\to \delta_\Gamma$. When this happens, the random walks $\mu_n$ are thought of as well spread in the group. Theorem~\ref{mainradius} shows that one can approach  the equality asymptotically in a more general inequality $-2\log \| \lambda_\Gamma(\mu)\| \geq \delta_\Gamma  n $, namely one can find a sequence of measures $\mu_n$ supported on $B_n$, so that $\frac{-2\log(\| \lambda_\Gamma(\mu_n)\| )}{n} \to \delta_\Gamma$.  
\medskip

We remark that the latter approximation is indeed stronger.

\begin{remark}
There exists sequence of measures $\mu_n$ on a free group on two generators, such that $\frac{h(\mu_n)}{l(\mu_n)} = \delta_\Gamma$, but $\frac{-2\log(\| \lambda_\Gamma(\mu_n)\| )}{l(\mu_n)} \to 0$
\end{remark}

To see this, consider the simple random walk on the free group on two generators $\Gamma=\left<a^{\pm 1},b^{\pm 1} \right>$ with the corresponding word metric. It is an easy exercise that the equality is achieved in both inequalities simultaneously. We will perturb the law $\mu$ preserving one of the equalities but not the other. 
\par 
We use the Markov stopping time (see \cite{Forghani}). For each $n\in \bbN$, we define the following cut set:
\[
C(n) = C_n^a \cup \{a\}
\]
where $C_n^a$ is the set of all reduced words of length $n$ that don't start with $a$. Markov stopping time creates a new law of random walk $\mu_n$. Intuitively, one can think of sample paths in the new random walk being the same paths as in the old one with the same distribution, but with rescaled time. Each unit of time in the new path corresponds to starting the walk from identity and hitting the cutting set. Forghani proved in \cite{Forghani} that both the entropy and the drift of the new random walk are obtained by multiplication of the initial entropy and drift by the expected value of the stopping time. Therefore for each $\mu_n$ the equality in the fundamental inequality still holds. 
\par 
It is easy to see that the spectral radius is bounded from below by $\frac{1}{4}$, regardless of $n$(test $\pi_\Gamma(\mu_n)$ against the characteristic function supported on powers of $a$), and since $l(\mu_n)\to \infty$ the claim follows.

\subsection{Optimal ergodic theorems}\hfill{}\par
Let $\Gamma<SL_2(\mathbb{Z})$, and $\pi_0$ be the Koopman representation on the 2-torus. By Theorem~\ref{toruskoopman} we have $\rho_{\pi_0}(n)=\rho_\lambda(n)$. Since the measures $\mu_n$ in Theorem~\ref{mainradius} are uniform measures on the shells $S_n$ in $\Gamma$, the operators $\pi_0(\mu_n)$ can be viewed as averaging operators, and we can reformualte Theorem~\ref{mainradius} as a quantitative  ergodic theorem.

\begin{cor}
	Let $\Gamma < \SL_2(\mathbb{Z})$. There exists $k>0$, so that if we denote the shells $S_n=B_n\setminus B_{n-k}\subset \Gamma$. 
	Then for any $f\in L^2(\mathbb{T}^2,m)$ we have
	\[
		\left\| \frac{1}{|S_n|} \sum_{g\in S_n} f(g.x) - \int_{\mathbb{T}^2} f dm \right\|_2 
		\leq n^2 e^{-\half\delta_\Gamma n +O(1)} \|f\|_2.
	\]
\end{cor}
From the inequalities in~\eqref{fundiq} the convergence rate can't be faster than $e^{- \half \delta_\Gamma n}$. This suggests that averaging over shells in $\Gamma$ produces the most optimal ergodic theorem for this action. 

\bibliographystyle{plain}
\bibliography{diophantinetorus}

\end{document}